\documentclass[12pt,reqno]{amsart}		
\usepackage[centertags]{amsmath}		
\usepackage{amsfonts}		
\usepackage{amssymb}		
\usepackage{amsthm}			
\usepackage[utf8]{inputenc}
\usepackage[T1]{fontenc}	
\usepackage{lmodern}
\usepackage[english]{babel}
\usepackage[autolanguage]{numprint}
\usepackage{dsfont}
\usepackage{textcomp}
\usepackage[a4paper,margin=2cm,headheight=14.5pt]{geometry}	
\usepackage{graphicx}		
\usepackage{fancyhdr}		
\usepackage[bookmarks,colorlinks,linkcolor=blue,citecolor=blue]{hyperref}
\usepackage[hypcap=false]{caption}
\usepackage{mathrsfs}			
\usepackage{calc}
\usepackage{color}

\theoremstyle{plain}
\newtheorem{thm}{Theorem}
\newtheorem{theorem}{Theorem}
\newtheorem{corollary}[thm]{Corollary}

\newtheorem{lemma}[thm]{Lemma}

\newtheorem{remark}[thm]{Remark}
\newtheorem{proposition}[thm]{Proposition}
\theoremstyle{definition}

\theoremstyle{remark}

\numberwithin{equation}{section}

\renewcommand{\textbf}[1]{\begingroup\bfseries\mathversion{bold}#1\endgroup} 

\newcommand{\N}{\mathbb{N}}

\renewcommand{\d}{\mathrm{d}}
\renewcommand{\div}{\operatorname{div}}
\title{Counter-example to a Kröger type spectral inequality}
\author{Salam Kouzayha \and Luc Pétiard}
\date{}
\title{Eigenvalues on the Laplacian with density}

\begin{document}

\begin{abstract}
Let $(M,g)$ be a compact Riemannian manifold with a boundary of class $\mathscr{C}^{1}$. We are interested in the spectrum of the weighted Laplacian on $M$ with Neumann boundary conditions. More precisely, given $\rho$ and $\sigma$ two positive functions on $M$, we study the eigenvalues of the equation $-\div(\sigma \nabla u)=\lambda\rho u$. Inspired by a recent work of B. Colbois and A. El Soufi \cite{colbois_soufi_2019}, we investigate upper bounds for the eigenvalues in the case where $\sigma=\rho^{\alpha}$, $\alpha>0$. We show that $\alpha = \frac{n-2}{n}$ plays a critical role in the estimation of the spectrum when the total mass of $\rho$ is fixed.
\end{abstract}

\maketitle

\section{Introduction}\ %

Let $(M,g)$ be a compact Riemannian manifold of dimension $n\geqslant2$ with a boundary of class $\mathscr{C}^{1}$. Let $\rho$ and $\sigma$ be two positive continuous functions defined on $M$. For all $u\in H^{1}(M)$, we denote by $\nabla u$ the gradient of $u$ with respect to the metric $g$ and we consider the Rayleigh quotient
\begin{equation}
\label{variationnel}
R_{(g,\sigma,\rho)}(u)=\frac{\int_{M}|\nabla u|^{2}\sigma \d V_{g}}{\int_{M}u^{2}\rho \d V_{g}}.
\end{equation}
 Its corresponding eigenvalues are given, for $k \in \N$, by
\[
\lambda_{k}^{g}(\rho,\sigma)=\inf\limits_{E_{k+1}\subset H^1(M)}\sup_{u\in E_{k+1}\setminus \{0\}}R_{(g,\sigma,\rho)}(u),
\]
where $E_{k+1}$ runs through the $(k+1)$-dimensional vector subspaces of $H^{1}(M)$ and $\d V_{g}$ is the volume element induced by the metric $g$.\\
Under some regularity conditions on $\rho$ and $\sigma$, $\lambda^{g}_{k}(\rho,\sigma)$ is the $k$-th eigenvalue of the problem
\begin{equation}
\label{luc3}
-\div(\sigma \nabla u)=\lambda\rho u\quad \text{in } M,
\end{equation}
with Neumann conditions on the boundary. 
When there is no risk of confusion, we use $\lambda_{k}(\rho,\sigma)$ instead of $\lambda_{k}^{g}(\rho,\sigma)$.
We said in the introduction that the spectrum of \eqref{luc3} is discrete, and can be ordered in a positive nondecreasing sequence that tends to infinity. Also, $\lambda_{0}(\rho,\sigma)=0$, the constant functions being eigenfunctions for $\lambda_{0}$. Although we can't find the eigenvalues explicitly in general, we can estimate them when we fix the total mass of $\rho$ to get some interesting inequalities.


What we do here is a continuation of several works that aimed to find a good choice of geometric restriction such that the supremum of the eigenvalues is bounded from above. The conformal spectrum is an important one and has been widely studied. Indeed, let us introduce the quantity
\begin{equation*}
\lambda_k^c(M,[g]) = \underset{g'\in [g]}{\sup} \lambda_k(M,g'),
\end{equation*}
then an upper bound on this quantity was found by Korevaar \cite[Theorem 0.4]{korevaar_1993}. A lower bound was later found by B.~Colbois, A.~El~Soufi in \cite[Corollary 1]{colbois_soufi_2003}. Together, these results read:
\begin{equation}
\label{encadrement_conforme}
n\omega_n^{\frac{2}{n}}k^{\frac{2}{n}}
\leqslant \lambda_k^c(M,[g])
\leqslant C([g])k^{\frac{2}{n}}
\end{equation}
where $\omega_n$ is the volume of the unit ball in dimension $n$ and $C([g])$ is a constant depending only on $n$ and on the conformal class of $g$. The lower bound is actually a corollary of an interesting result concerning the gap between two extremal eigenvalues, which states that
\begin{equation*}
\lambda_{k+1}^c(M,[g])-\lambda_k^c(M,[g]) \geqslant n^{\frac{n}{2}}\omega_n.
\end{equation*} 
 
As the problem of the Laplacian with densities is very general, we chose to restrict ourselves to some particular class of densities. In this chapter, we are interested in the problem \eqref{luc3} when $\sigma=\rho^{\alpha}$, and $\alpha\in [0,1]$. The results we obtain on the spectrum $\lambda_{k}^g(\rho,\rho^{\alpha})$ are supported by the three following important theorems, that led our motivation and intuition: 
\begin{itemize}
\item
When $\alpha = 0$, A.~El~Soufi and B.~Colbois proved in \cite[Corollary 4.1]{colbois_soufi_2019} that, for any compact Riemannian manifold $(M,g_{0})$ with density $\rho$, and for any metric $g$ conformal to $g_{0}$ such that $\int_{M} \rho \d V_{g}=|M|_{g}$, one has
\[
	\lambda_{k}(\rho,1)|M|_{g}^{\frac{2}{n}}\leqslant C_{n}k^{\frac{2}{n}}+D_{n}|M|_{g_{0}}^{\frac{2}{n}},
\]
where $C_n$ and $D_n$ are constants depending only on $n$.
\item
Another important problem is when $\alpha = 1$, that is, $\sigma=\rho$. The associated operator is called the Witten Laplacian and was treated by B.~Colbois, A.~El~Soufi and A.~Savo. In their work \cite[Theorem 5.2]{colbois_soufi_savo_2015}, they proved that, contrary to the previous cases, one cannot bound the eigenvalues from above on all manifolds. Indeed, on a compact manifold of revolution endowed with the Gaussian radial density $\rho_{m}=e^{-m|x|^{2}}$, one has for $m$ large enough,
\[
\lambda_{1}(\rho_{m},\rho_{m})\geqslant m.
\]
\item A further important case is when $\alpha = \frac{n-2}{n}$, and actually comes from several works on the conformal spectrum. We get an upper bound thanks to the work of A.~Hassannezhad \cite[Theorem 1.1]{hassannezhad_2011}, where she finds an inequality for the classical Laplace equation:
\[
\lambda_{k}^g(1,1)|M|_{g}^{\frac{2}{n}}\leqslant A_{n}k^{\frac{2}{n}}+B_{n}V([g])^{\frac{2}{n}}
\]
where $A_{n}$ and $B_{n}$ are two constants which only depend on $n$, and $V\left([g]\right)$ is the geometric quantity defined by:
\[
V([g])=\inf\left\{|M|_{g'}:g'\mbox{ is conformal to } g \mbox{ and } \text{Ricci}(g')\geqslant-(n-1)g' \right\}.
\]
If we now take $\rho$ a positive continuous function on $M$ satisfying $\int_{M}\rho\d V_g = |M|_g$, one can easily check that $|M|_{\rho^{\frac{2}{n}}g} = \int_{M}\rho\d V_g = |M|_g$ and $\lambda_{k}^{\rho^{\frac{2}{n}}g}(1,1)=\lambda_{k}^{g}(\rho,\rho^{\frac{n-2}{n}})$.
Since $g$ and $\rho^{\frac{2}{n}}g$ are conformal, then $V([\rho^{\frac{2}{n}}g]) = V([g])$.
Thus we obtain
\[
\lambda_{k}^{g}\left(\rho,\rho^{\frac{n-2}{n}}\right)|M|_{g}^{\frac{2}{n}}
=\lambda_{k}^{\rho^{\frac{2}{n}}g}(1,1)|M|_{\rho^{\frac{2}{n}}g}^{\frac{2}{n}}
\leqslant A_{n}k^{\frac{2}{n}}+B_{n}V([g])^{\frac{2}{n}},
\]
which gives us an upper bound for $\lambda_{k}^{g}\left(\rho,\rho^{\alpha}\right)$ in the case where $\alpha$ is equal to $\frac{n-2}{n}\cdot$
\end{itemize}



For the following we will denote by $\lambda_{k,\alpha}^{*}(M,g)$ the supremum of $\lambda_{k}^g(\rho,\rho^{\alpha})$ on the set of all densities $\rho$ satisfying $\int_{M}\rho \d V_{g}=|M|_{g}$, that is,
\begin{equation*}
\lambda^{*}_{k,\alpha}(M,g)=\underset{\rho}{\sup}\left\{\lambda_{k}^{g}(\rho,\rho^{\alpha}),\int_{M}\rho \d V_{g}=|M|_{g}\right\}.
\end{equation*}
We believe that a uniform lower bound of the same type as in equation \eqref{encadrement_conforme} does not exist for this supremum if $\alpha \in \left[0,\frac{n-2}{n}\right]$. Indeed, when $\alpha = 0$, we can always find a $1$-parameter family of metrics of volume $1$ making it very small. The construction of these metrics can be found in \cite[Theorem 5.1]{colbois_soufi_2019}, and should be generalised for $\alpha \in \left(0,\frac{n-2}{n}\right]$. More generally, it is still an open question to know if there exists a non-negative bound for the spectral gap with densities, that is:
\begin{equation*}
\lambda^{*}_{k+1,\alpha}(M,g)-\lambda^{*}_{k,\alpha}(M,g) > K(g)
\end{equation*}
where $K(g)\geqslant 0$ is a constant depending only on the metric.

An interesting question is to study the behaviour of the spectrum for different values of $\alpha$. In fact, through the theorems \ref{thm_borne_sup} and \ref{thm_sup_infini}, we will see that $\alpha = \frac{n-2}{n}$ plays a significant role.
To highlight this, we first show $\lambda^{*}_{k,\alpha}(M,g)$ is bounded when $\alpha$ runs over the interval $\left(0,\frac{n-2}{n}\right)$:

\begin{theorem}
\label{thm_borne_sup}
Let $n \geqslant 2$. Then there exist $A_n$, $B_n$ two positive constants such that, for any bounded domain $M$ with boundary of class $\mathscr{C}^{1}$ of a complete Riemannian manifold $(\tilde{M},\tilde{g}_{0})$ of dimension $n$, verifying $\text{Ricci}(\tilde{g}_{0})\geqslant -(n-1)\tilde{g}_{0}$, we have, for every metric $g$ conformal to $g_{0}:=\left(\tilde{g}_{0}\right)_{|_{M}}$, for all $\alpha \in \left(0,\frac{n-2}{n}\right)$, and every density $\rho$ such that $\int_{M}\rho \d V_{g}=|M|_{g}$:
\begin{equation*}
	\lambda_{k}(\rho,\rho^{\alpha})|M|_{g}^{\frac{2}{n}}\leqslant A_{n}k^{\frac{2}{n}}+B_{n}|M|_{g_{0}}^{\frac{2}{n}}.
\end{equation*}
\end{theorem}

\begin{remark}\ %

Let $g$ be a metric defined on $M$ such that $g_0 = \text{ric}_0 g$, for some $\text{ric}_0 > 0$. Then if $\text{Ricci}\left(g_0\right) \geqslant -(n-1)g_0$, we have $\text{Ricci}(g) \geqslant -(n-1)\text{ric}_0 g$ and $|M|_{g_0} = \text{ric}_0^{\frac{n}{2}}|M|_g$. The following corollary follows:
\end{remark}

\begin{corollary}\ %
\label{corollaire_thm_borne_sup}

Let $M$ be a bounded domain with boundary of class $\mathscr{C}^{1}$ of a complete Riemannian manifold $(\tilde{M},\tilde{g})$ of dimension $n\geqslant 2$ such that $\text{Ricci}(\tilde{g})\geqslant -(n-1)\text{ric}_0 \tilde{g}$ and let $g=\tilde{g}_{|_{M}}$. For every density $\rho$ such that $\int_{M}\rho \d V_{g}=|M|_{g}$, we have:

\begin{equation*}
	\lambda_{k}(\rho,\rho^{\alpha})|M|_{g}^{\frac{2}{n}}\leqslant A_{n}k^{\frac{2}{n}}+B_{n}\text{ric}_0|M|_{g}^{\frac{2}{n}}
\end{equation*}
where $\alpha\in \left(0,\frac{n-2}{n}\right)$ and $ric_0 > 0$.

\end{corollary}

In the last section we show that $\lambda^{*}_{k,\alpha}(M,g)$ is infinite when $\alpha$ belongs to $\left(\frac{n-2}{n},1\right)$, and $(M,g)$ is a manifold of revolution.

\begin{remark}\ %

A natural question emerges: what happens when $\alpha > 1$? We believe the supremum
is not uniformly bounded and can be infinite for some manifolds. However, our attempts
in that direction remained fruitless.

\end{remark}

%
%

%
%

\section{Bounding the eigenvalues from above}\ %

In this section, we suppose that  $\alpha \in \left(0,\frac{n-2}{n}\right)$.
We define $M$ as a bounded submanifold of dimension $n$ of a complete Riemannian manifold $(\tilde{M},\tilde{g}_{0})$ with Ricci curvature bounded from below and $\rho$ as a positive continuous function on $M$. Using the same argument of A. El Soufi and B. Colbois in \cite[Theorem 4.1]{colbois_soufi_2019}, we are able to maximise the eigenvalues $\lambda_k(\rho,\rho^{\alpha})$ on the class of metrics conformal to $g_{0}=\tilde{g_{0}}_{|_{M}}$ under the preservation of the total mass of $\rho$.

\bigskip
The following lemma was first introduced by Asma Hassannezhad in \cite[Theorem 2.1]{hassannezhad_2011}.
It is based on two important technical results of Grigor'yan, Netrusov, Yau (see \cite[section 3]{grigor}) and Colbois, Maerten \cite[Corollary 2.3]{colbois_maerten_2008}.
We say that a metric measured space $(X,d,\mu)$ satisfies the $(2,N;1)$-covering property if every ball of radius $r\leqslant 1$ can be covered by $N$ balls of radius $\frac{r}{2}$.

\begin{lemma}
\label{coveringlemma_intro}
Let $(X,d,\mu)$ be a complete and locally compact metric measured space, where $\mu$ is a non-atomic measure. Suppose $X$ satisfies the $(2,N;1)$-covering property. Then for every $k\in\mathbb{N}^{*}$, there exists a family of $4(k+1)$ measurable sets 
$F_j,G_j$ with the following properties:
  \begin{enumerate}
    \item $F_{j}\subset G_{j}$.
    \item The $G_{j}$'s are mutually disjoint.
		\item $\mu(F_{j})\geqslant \frac{\mu(X)}{c^{2}(k+1)}$ with $c=c(N)$ a constant which depends only on $N$.
		\medskip
    \item The $(F_{j},G_{j})$'s are of the following two types:
		 \begin{itemize}
		 	\item For all $j$, $F_{j}$ is an annulus $A=\{r\leqslant d(x,a)<R\}$,
		 	
		 	and $G_{j}=2A=\left\{\frac{r}{2}\leqslant d(x,a)<2R\right\}$ with $0\leqslant r \leqslant R$ and $2R<1$, or:
		 	\item For all $j$, $F_{j}$ is an open set $F\subset M,$
		 	
		 	and $G_{j}=F^{r_{0}}=\{x\in M;d(x,F)\leqslant r_{0}\}$ with $r_{0}=\frac{1}{1600}$.
		 \end{itemize}
     \end{enumerate}
\end{lemma}

\begin{proof}[Proof of Theorem \ref{thm_borne_sup}]

To prove Theorem \ref{thm_borne_sup}, we construct $k+1$ test functions on $M$ with disjoint supports and controlled Rayleigh quotients.
	The idea is to use the covering property that was applied by A. Hassannezhad in \cite[Theorem 2.1]{hassannezhad_2011} to find $k+1$ functions defined on $M$ with disjoint supports and controlled Rayleigh quotients. Let $\mu$ be the measure defined by its volume element $\mu=\rho \d V_{g}$. Since $\text{Ric}(\tilde{g}_0)\geqslant -(n-1)\tilde{g}_0$, the metric measured space $(\tilde{M},\tilde{d}_0,\mu)$ satisfies the (2;$N$;1)-covering property for some fixed $N$ (see \cite{hassannezhad_2011}), and we can apply Lemma \ref{coveringlemma_intro}.
%
%
Define the distance $d_{0}$ as the restriction on $M$ of the distance $\tilde{d}_0$ induced by $\tilde{g}_0$. We are going to treat two cases separately:

\medskip
	\textbf{First case:}
	$F_{j}$ is a generic annulus $A$ and $G_j = 2A$.
	\begin{center}
		\def\svgscale{0.8}
\begingroup%
  \makeatletter%
  \providecommand\color[2][]{%
    \errmessage{(Inkscape) Color is used for the text in Inkscape, but the package 'color.sty' is not loaded}%
    \renewcommand\color[2][]{}%
  }%
  \providecommand\transparent[1]{%
    \errmessage{(Inkscape) Transparency is used (non-zero) for the text in Inkscape, but the package 'transparent.sty' is not loaded}%
    \renewcommand\transparent[1]{}%
  }%
  \providecommand\rotatebox[2]{#2}%
  \ifx\svgwidth\undefined%
    \setlength{\unitlength}{512.06469193bp}%
    \ifx\svgscale\undefined%
      \relax%
    \else%
      \setlength{\unitlength}{\unitlength * \real{\svgscale}}%
    \fi%
  \else%
    \setlength{\unitlength}{\svgwidth}%
  \fi%
  \global\let\svgwidth\undefined%
  \global\let\svgscale\undefined%
  \makeatother%
  \begin{picture}(1,0.40566455)%
    \put(0,0){\includegraphics[width=\unitlength,page=1]{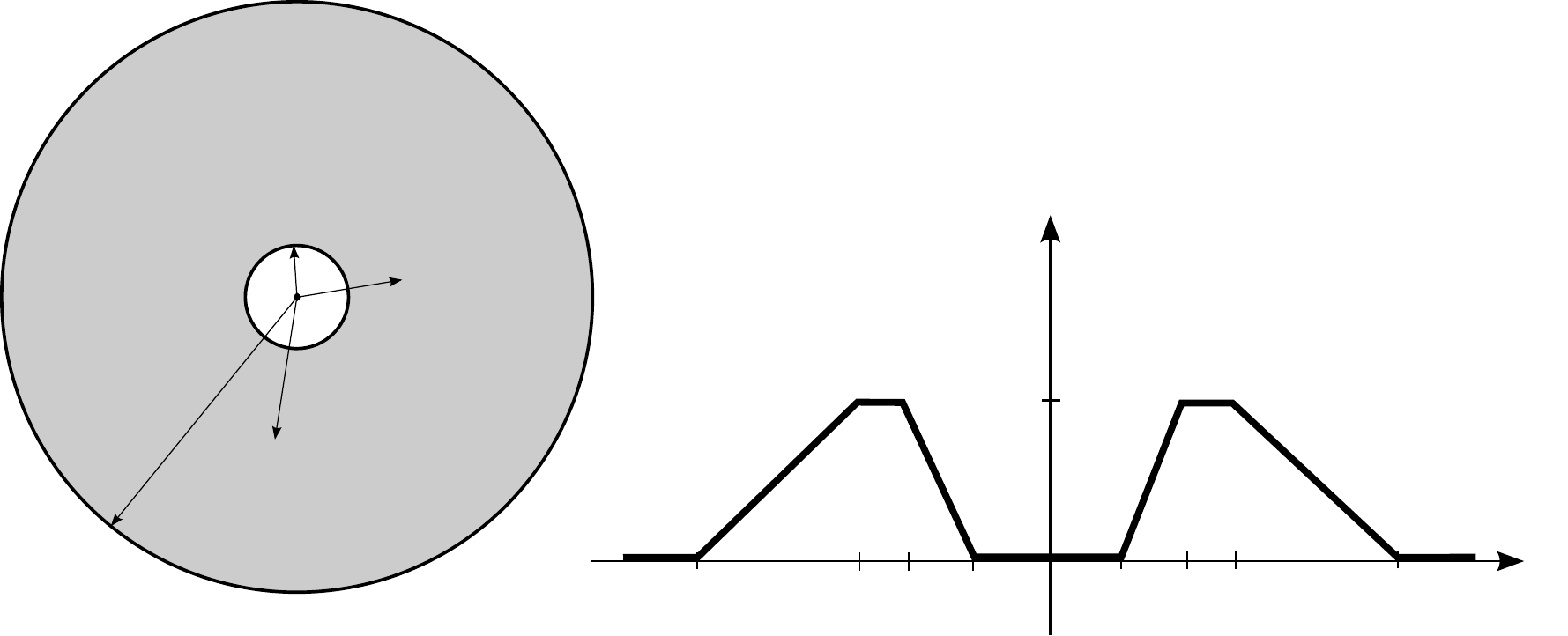}}%
    \put(0.23484482,0.21811271){\color[rgb]{0,0,0}\makebox(0,0)[lt]{\begin{minipage}{0.01732754\unitlength}\raggedright $r$\end{minipage}}}%
    \put(0,0){\includegraphics[width=\unitlength,page=2]{img_anneaux_bis.pdf}}%
    \put(0.17328071,0.22951225){\color[rgb]{0,0,0}\makebox(0,0)[lb]{\smash{$_{\frac{r}{2}}$}}}%
    \put(0.70570484,0.01436765){\color[rgb]{0,0,0}\makebox(0,0)[lb]{\smash{$\frac{r}{2}$}}}%
    \put(0.93750762,0.0596899){\color[rgb]{0,0,0}\makebox(0,0)[lb]{\smash{$d_{0}(x,a)$}}}%
    \put(0.9214762,0.01760093){\color[rgb]{0,0,0}\makebox(0,0)[lb]{\smash{$1$}}}%
    \put(0,0){\includegraphics[width=\unitlength,page=3]{img_anneaux_bis.pdf}}%
    \put(0.87574926,0.01761098){\color[rgb]{0,0,0}\makebox(0,0)[lb]{\smash{$2R$}}}%
    \put(0.75101476,0.02538925){\color[rgb]{0,0,0}\makebox(0,0)[lb]{\smash{$r$}}}%
    \put(0.18357644,0.16087648){\color[rgb]{0,0,0}\makebox(0,0)[lb]{\smash{$R$}}}%
    \put(0.15786632,0.34368758){\color[rgb]{0,0,0}\makebox(0,0)[lb]{\smash{$2A$}}}%
    \put(0.13357548,0.27389894){\color[rgb]{0,0,0}\makebox(0,0)[lb]{\smash{$A$}}}%
    \put(0.11610827,0.1044561){\color[rgb]{0,0,0}\makebox(0,0)[lb]{\smash{$2R$}}}%
    \put(0.7797166,0.01956304){\color[rgb]{0,0,0}\makebox(0,0)[lb]{\smash{$R$}}}%
    \put(0.65158436,0.14336281){\color[rgb]{0,0,0}\makebox(0,0)[lb]{\smash{$1$}}}%
  \end{picture}%
\endgroup%

		\captionof{figure}{Behaviour of $u_{A}$}
	\end{center}
	Define the function $u_{A}$ supported in $G_{j}=2A$ by:
	\[
	u_{A}(x) = \left\{
	\begin{array}{lll}
	\frac{2}{r}d_{0}(x,a)-1\quad&\mbox{if }& \frac{r}{2}<d_{0}(x,a)<r \\
	1\quad&\mbox{if }& r<d_{0}(x,a)<R\\
	2-\frac{1}{R}d_{0}(x,a)\quad&\mbox{if }& R<d_{0}(x,a)<2R.
	\end{array}
	\right.
	\]
	Since $u=1$ on $A$, then we have 
	\begin{equation}
	\label{luc10}
\int_{M}u_{A}^{2}\rho \d V_{g}\geqslant \int_{A}u_{A}^{2}\rho \d V_{g}= \int_{A}\rho \d V_{g}=\mu(A)\geqslant \frac{\mu(M)}{c^{2}(k+1)}.
\end{equation}
On the other hand, using Hölder's inequality repeatedly on the integral $\int_{M}|\nabla^{g} u_{A}|^{2}\rho^{\alpha}\d V_{g}$ and the fact that the generalised Dirichlet energy $\int_{2A}|\nabla^{g} u_{A}|^{n}\d V_{g}$ is a conformal invariant, we get
	\begin{align*}
\int_{M}|\nabla^{g} u_{A}|^{2}\rho^{\alpha}\d V_{g}=\int_{2A}|\nabla^{g} u_{A}|^{2}\rho^{\alpha}\d V_{g}
	&\leqslant \left(\int_{2A}|\nabla^{g} u_{A}|^{n}\d V_{g}\right)^{\frac{2}{n}}\left(\int_{2A}\rho^{\frac{n\alpha}{n-2}}\d V_{g}\right)^{\frac{n-2}{n}} \\
	&\leqslant \left(\int_{2A}|\nabla^{g} u_{A}|^{n}\d V_{g}\right)^{\frac{2}{n}} \left(\int_{2A}\rho \d V_{g}\right)^{\alpha}\left(\int_{2A}1\d V_{g}\right)^{\frac{n-2}{n}-\alpha}\\
	&= \left(\int_{2A}|\nabla^{g_{0}} u_{A}|^{n} \d V_{g_{0}}\right)^{\frac{2}{n}} \mu(2A)^{\alpha}|2A|_{g}^{\frac{n-2}{n}-\alpha}.
	\end{align*}
	Since 
	\[
	|\nabla^{g_{0}}u_{A}| = \left\{
	\begin{array}{lll}
	\frac{2}{r}\quad&\mbox{if }& \frac{r}{2}<d_{0}(x,a)<r \\
	0\quad&\mbox{if }& r<d_{0}(x,a)<R\\
	\frac{1}{R}\quad&\mbox{if }& R<d_{0}(x,a)<2R,
	\end{array}
	\right.
	\] 
we obtain:

\begin{equation}
\label{ineg_volume_boules}
\int_{2A}|\nabla^{g_{0}} u_{A}|^{n}\d V_{g_{0}}
\leqslant \left(\frac{2}{r}\right)^{n}|B(a,r)|_{g_{0}}+\left(\frac{1}{R}\right)^{n}\left|B(a,2R)\right|_{g_{0}}.
\end{equation}

But $r\leqslant 2R \leqslant 1$ and $\text{Ric}(g_{0})\geqslant -(n-1)g_{0}$. It is also a well known fact that thanks to the Bishop-Gromov comparison Theorem, one can compare the volume of any ball in $(M,g_0)$ to the volume of the ball of same radius in the hyperbolic space of constant curvature $-1$.
More information can be found in \cite[Theorem 3.1]{zhu_1997}.
We deduce that the right-hand side of inequality \eqref{ineg_volume_boules} is bounded from above by a quantity $\tilde{A}_n$ depending only on $n$.
Hence the following inequality holds:
	\begin{equation}
	\label{luc11}
	\int_{M}|\nabla^{g} u_{A}|^{2}\rho^{\alpha}\d V_{g}\leqslant \left(\tilde{A}_n\right)^{\frac{2}{n}}\mu(2A)^{\alpha}|2A|_{g}^{\frac{n-2}{n}-\alpha}.
	\end{equation}
	
	From \eqref{luc10} and \eqref{luc11}, we deduce that the Rayleigh quotient is bounded as follows:
	\begin{equation}
	\label{luc7}
	R_{(g,\rho,\rho^{\alpha})}(u_{A})=\frac{\int_{M}|\nabla^{g} u_{A}|^{2}\rho^{\alpha}\d V_{g}}{\int_{M}u_{A}^{2}\rho \d V_{g}}\leqslant \left(\tilde{A}_n\right)^{\frac{2}{n}}c^2\frac{\mu(2A)^{\alpha}|2A|_{g}^{\frac{n-2}{n}-\alpha}}{\mu(M)}(k+1).
	\end{equation}

	\medskip
	\textbf{Second case:} $F_{j}$ is a generic subset $V$ of $M$ and $G_{j}= V^{r_{0}}$, the set at distance $r_0$ from $V$.
	\begin{center}
		\def\svgscale{0.25}
\begingroup%
  \makeatletter%
  \providecommand\color[2][]{%
    \errmessage{(Inkscape) Color is used for the text in Inkscape, but the package 'color.sty' is not loaded}%
    \renewcommand\color[2][]{}%
  }%
  \providecommand\transparent[1]{%
    \errmessage{(Inkscape) Transparency is used (non-zero) for the text in Inkscape, but the package 'transparent.sty' is not loaded}%
    \renewcommand\transparent[1]{}%
  }%
  \providecommand\rotatebox[2]{#2}%
  \ifx\svgwidth\undefined%
    \setlength{\unitlength}{538.85227742bp}%
    \ifx\svgscale\undefined%
      \relax%
    \else%
      \setlength{\unitlength}{\unitlength * \real{\svgscale}}%
    \fi%
  \else%
    \setlength{\unitlength}{\svgwidth}%
  \fi%
  \global\let\svgwidth\undefined%
  \global\let\svgscale\undefined%
  \makeatother%
  \begin{picture}(1,1.0313037)%
    \put(0,0){\includegraphics[width=\unitlength]{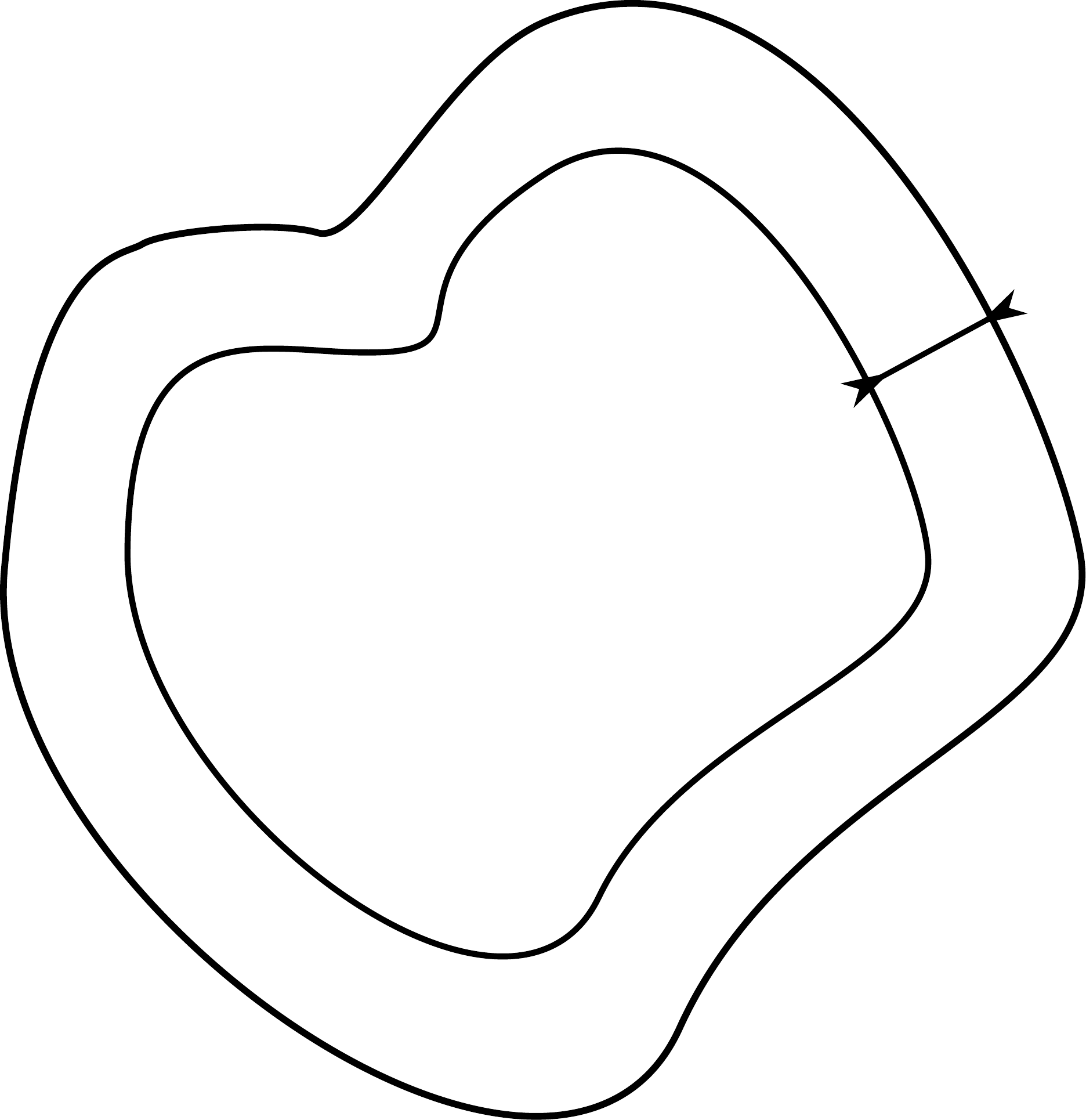}}%
    \put(0.85807032,0.7050196){\color[rgb]{0,0,0}\makebox(0,0)[lt]{\begin{minipage}{0.18243614\unitlength}\raggedright $r_0$\end{minipage}}}%
    \put(0.43382821,0.75816786){\color[rgb]{0,0,0}\makebox(0,0)[lt]{\begin{minipage}{0.12803047\unitlength}\raggedright $F$\end{minipage}}}%
    \put(0.77084013,0.93561981){\color[rgb]{0,0,0}\makebox(0,0)[lt]{\begin{minipage}{0.24468052\unitlength}\raggedright $F^{r_0}$\end{minipage}}}%
  \end{picture}%
\endgroup%

		\captionof{figure}{Behavior of $u_{V}$}
	\end{center}
	We define the function $u_{V}$ supported in $V^{r_{0}}$ by:
	\[
	u_{V}(x) = \left\{
	\begin{array}{lll}
	1\quad						&\mbox{if }& x\in V \\
	1-\frac{1}{r_{0}}d_0(x,V)\quad&\mbox{if }& x\in V^{r_{0}}\backslash V \\
	0							&\mbox{if }& x\in M\setminus V^{r_{0}}
	\end{array}
	\right.
	\]
Here we have:
	\[
	\int_{M}u_{V}^{2}\rho \d V_{g}\geqslant \int_{V}\rho \d V_{g}
	=\mu(V)\geqslant\frac{\mu(M)}{c^{2}(k+1)}
	\]
	and we also use Hölder as in the previous case:
\begin{align*}
\int_{M}|\nabla^{g}u_{V}|^{2}\rho^{\alpha}\d V_{g}
=\int_{V^{r_{0}}}|\nabla^{g}u_{V}|^{2}\rho^{\alpha}\d V_{g}
	&\leqslant \left(\int_{V^{r_{0}}}|\nabla^{g} u_{V}|^{n}\d V_{g}\right)^{\frac{2}{n}}
\left(\int_{V^{r_{0}}}\rho^{\frac{n\alpha}{n-2}}\d V_{g}\right)^{\frac{n-2}{n}}  \\
	&\leqslant \left(\int_{V^{r_{0}}}|\nabla^{g} u_{V}|^{n}\d V_{g}\right)^{\frac{2}{n}}\left(\int_{V^{r_{0}}}\rho \d V_{g}\right)^{\alpha}\left(\int_{V^{r_{0}}}1\d V_{g}\right)^{\frac{n-2}{n}-\alpha}  \\
			&= \left(\int_{V^{r_{0}}}|\nabla^{g_{0}} u_{V}|^{n}\d V_{g_{0}}\right)^{\frac{2}{n}} \mu(V^{r_{0}})^{\alpha}|V^{r_{0}}|_{g}^{\frac{n-2}{n}-\alpha}.
\end{align*}

	Since
	\[
	|\nabla^{g_{0}}u_{V}(x)| = \left\{
	\begin{array}{lll}
	0 \quad &\mbox{if }& x\in V\\
	\frac{1}{r_{0}}\quad &\mbox{if }& x\in V^{r_{0}}\backslash V
	\end{array}
	\right.
	\]
	we get
	\[
	\int_{V^{r_{0}}}|\nabla^{g_{0}} u_{V}|^{n}\d V_{g_{0}}
	=\frac{1}{r_{0}^{n}} |V^{r_{0}}\backslash V|_{g_{0}}
	\leqslant \frac{1}{r_{0}^{n}} |V^{r_{0}}|_{g_{0}},
	\]
	and then
	\[\int_{M}|\nabla^{g}u_{V}|^{2}\rho^{\alpha}\d V_{g}\leqslant \frac{|V^{r_{0}}|_{g_{0}}^{\frac{2}{n}}\mu(V^{r_{0}})^{\alpha}\left|V^{r_{0}}\right|_{g}^{\frac{n-2}{n}-\alpha}}{r_{0}^{2}}.\]

	We conclude that
	\begin{equation}
	\label{luc6}
	R_{(g,\rho,\rho^{\alpha})}(u_{V})\leqslant B_{n}\frac{|V^{r_{0}}|_{g_{0}}^{\frac{2}{n}}\mu(V^{r_{0}})^{\alpha}|V^{r_{0}}|_{g}^{\frac{n-2}{n}-\alpha}}{\mu(M)}(k+1),
	\end{equation}
	where $B_{n}=\frac{c^{2}}{r_{0}^{2}}$ depends only on $n$.
	Thus we were able to bound the two Rayleigh quotients by some quantities. We are going to show that these quantities can actually be bounded in the following way.
	
	We use the next lemma, whose proof is given at the end of this section.
	\begin{lemma}
	\label{lemme_3_mesures}
Let $M$ be a Riemannian manifold and let $\nu_1, \nu_2, \nu_3$ be any measures on $M$. Take a collection of $K$ disjoint open subsets $(U_i)_i$ in $M$. If $K \geqslant 4k+1$ for some $k \in \N^*$, then there exist $k+1$ open subsets in this collection such that they satisfy the three conditions below:
\begin{equation*}
\nu_1(U_i) \leqslant \frac{\nu_1(M)}{k+1}, \ \ \ \ \ \ \ \ \ \ \
\nu_2(U_i) \leqslant \frac{\nu_2(M)}{k+1} \ \ \ \ \ \text{and} \ \ \ \ \ \
\nu_3(U_i) \leqslant \frac{\nu_3(M)}{k+1}.
\end{equation*}
	\end{lemma}

	
	As the $4k+4$ sets $G_{j}$ are disjoint, we deduce there exist $k+1$ sets among them satisfying the following three inequalities:
	\[
	|G_{j}|_{g_{0}}\leqslant \frac{|M|_{g_{0}}}{k+1}, \quad \quad |G_{j}|_{g}\leqslant \frac{|M|_{g}}{k+1}\quad \text{and} \quad\mu(G_{j})\leqslant \frac{\mu(M)}{k+1}.
	\]
	Using the previous estimates of the Rayleigh quotients \eqref{luc7} and \eqref{luc6}, we obtain $k+1$ disjointly supported functions $u_{j}$, $j=1,\dots,k+1$, with Rayleigh quotients satisfying either
	\begin{align*}
R_{(g,\rho,\rho^{\alpha})}(u_{j})
	&\leqslant \left(\tilde{A}_n\right)^{\frac{2}{n}}c^2
	\frac{\mu(G_{j})^{\alpha}|G_{j}|_{g}^{\frac{n-2}{n}-\alpha}}{\mu(M)}(k+1) \\
	&\leqslant \left(\tilde{A}_n\right)^{\frac{2}{n}}c^2
	\frac{\mu(M)^{\alpha}|M|_{g}^{\frac{n-2}{n}-\alpha}}{(k+1)^{\alpha}(k+1)^{\frac{n-2}{n}-\alpha}\mu(M)}(k+1) \\
	&= \left(\tilde{A}_n\right)^{\frac{2}{n}}c^2\left(\frac{k+1}{|M|_{g}}\right)^{\frac{2}{n}}\\
	&\leqslant A_n\left(\frac{k}{|M|_{g}}\right)^{\frac{2}{n}},
	\end{align*}
	or 
	\begin{align*}
R_{(g,\rho,\rho^{\alpha})}(u_{j})
	&\leqslant B_n\frac{|G_{j}|_{g_{0}}^{\frac{2}{n}}\mu(G_{j})^{\alpha}|G_{j}|_{g}^{\frac{n-2}{n}-\alpha}}{\mu(M)}(k+1) \\
	&\leqslant B_{n}\frac{|M|_{g_{0}}^{\frac{2}{n}}\mu(M)^{\alpha}|M|_{g}^{\frac{n-2}{n}-\alpha}}{(k+1)^{\frac{2}{n}}(k+1)^{\alpha}(k+1)^{\frac{n-2}{n}-\alpha}\mu(M)}(k+1) \\
	&= B_{n}\left(\frac{|M|_{g_{0}}}{|M|_{g}}\right)^{\frac{2}{n}}.
	\end{align*}
	Consequently $\lambda_{k}^{g}(\rho,\rho^{\alpha})$ is bounded above and
\begin{equation*}
\lambda^{g}_{k}(\rho,\rho^{\alpha})|M|_{g}^{\frac{2}{n}}\leqslant A_{n}k^{\frac{2}{n}}+B_{n}|M|_{g_{0}}^{\frac{2}{n}}.
\end{equation*}
\end{proof}

\begin{proof}[Proof of Lemma \ref{lemme_3_mesures}.]

First, notice that at most $k$ subsets are such that $\nu_1(U_i)>\frac{\nu_1(M)}{k+1}$. Indeed, assume there exist $k+1$ subsets verifying $\nu_1(U_i) > \frac{\nu_1(M)}{k+1}$. Then the volume for $\nu_1$ of these (disjoint) subsets would be greater than the volume of $M$, which is a contradiction.

Now we know we can work on a collection of $K-k$ sets $U_i$ satisfying $\nu_1(U_i) \leqslant \frac{\nu_1(M)}{k+1}$. We repeat the idea to take $K-2k$ sets from this collection that satisfy $\nu_1(U_i) \leqslant \frac{\nu_1(M)}{k+1}$ and $\nu_2(U_i) \leqslant \frac{\nu_2(M)}{k+1}$. We repeat again, and finally extract from these the $K-3k$ smallest sets for the measure $\nu_3$. As $K\geqslant 4k+1$, we have $K-3k \geqslant k+1$, which finishes the proof.

	\end{proof}

\section{Construction of densities with large \texorpdfstring{$\lambda_{1}$}{Lg}}\ %
 
 Now that we have seen we can bound $\lambda_{k}^g(\rho,\rho^{\alpha})$ for $\alpha \in \left(0,\frac{n-2}{n}\right)$, we are going to show that for $\alpha \in \left(\frac{n-2}{n},1\right)$, the supremum $\lambda_{1,\alpha}^{*}(M,g)$ can be equal to $+\infty$ for a certain type of manifold.
The reader can refer to \cite[Section 5]{colbois_soufi_savo_2015} for the definition of a manifold of revolution.
 
\begin{theorem}
	\label{thm_sup_infini}
	Let $\Omega$ be a manifold of revolution of dimension $n \geqslant 2$.
If $\alpha \in \left(\frac{n-2}{n},1\right)$, then
	\begin{equation*}
	\lambda^{*}_{k,\alpha}(\Omega)=+\infty .
	\end{equation*}
\end{theorem}

\begin{proof}
Without loss of generality, we assume $0\in \Omega$.
For all $m\geqslant 1$, we define the radial density function $\rho_{m}$ by
	\[
	\rho_m(x) = e^{-m|x|^2}.
	\]
	As $\rho_m \leqslant 1$ and $\alpha < 1$, we get $\rho_m^{\alpha} \geqslant\rho_m$. Then
	\[
	\lambda_{1}(\rho_m,\rho_m^{\alpha}) \geqslant \lambda_{1}(\rho_m,\rho_m).
	\]
According to A. Savo, A. El Soufi, and B. Colbois (see \cite[Theorem 5.2]{colbois_soufi_savo_2015}), we know that in dimension larger than $2$, there exists an $m_0$ such that for all $m\geqslant m_0$, $\lambda_{1} (\rho_m,\rho_m) \geqslant m$.
	Thus for $m\geqslant m_0$, 
	\begin{equation}
	\label{luc1}
	\lambda_{1}(\rho_m,\rho_m^{\alpha}) \geqslant m.
	\end{equation}
	Let $\tilde{\rho}_{m}=\frac{\rho_{m}|\Omega|}{\int_{\Omega}\rho_{m}\d x}\cdot$ It is clear that $\tilde{\rho}_{m}$ is a continuous bounded function on $\Omega$ with $\int_{\Omega}\tilde{\rho}_{m}\d x=~|\Omega|$.
	Thanks to the variational characterisation \eqref{variationnel}, we get
	\begin{equation}
	\label{luc2}
	\lambda_{1}(\tilde{\rho}_{m},\tilde{\rho}_{m}^{\alpha})
	= \underset{u\in E_{1}}{\inf}
	 \frac{\int_\Omega|\nabla u|^2 \left(\frac{\rho_{m}|\Omega|}{\int_{\Omega}\rho_{m}\d x}\right)^{\alpha} \d x}{\int_\Omega u^2 \left(\frac{\rho_{m}|\Omega|}{\int_{\Omega}\rho_{m}\d x}\right) \d x}
	 =	\lambda_{1}(\rho_m,\rho_m^{\alpha})
	\left(\frac{1}{|\Omega|}\int_{\Omega} \rho_{m}\d x\right)^{1-\alpha} .
	\end{equation}
	Using \eqref{luc1} and \eqref{luc2}, we get
	\[\lambda_{1}(\tilde{\rho}_{m},\tilde{\rho}_{m}^{\alpha})\geqslant m \left(\frac{1}{|\Omega|}\int_{\Omega}\rho_{m}\d x\right)^{1-\alpha}.\]
	It remains to estimate $\left(\int_{\Omega} \rho_m \d x\right)^{1-\alpha}$.
	
\begin{lemma}
\label{minoration_integrale}
	 For $m$ large enough,
	\begin{equation*}
	\int_{\Omega} e^{-m|x|^2}\d x
	\geqslant
	e^{-n} m^{-\frac{n}{2}}.
	\end{equation*}
\end{lemma}
	
\begin{proof}
	As $0$ is in $\Omega$, there exists $L>0$ such that $\Omega$ contains the $n$-square $(-L,L)^n$.
	
	Therefore,
	\begin{equation*}
	\int_{\Omega} e^{-m|x|^2}\d x
	\geqslant
	\left(\int_{-L}^L e^{-m t^2} \d t\right)^n.
	\end{equation*}
	
Notice that $t\mapsto e^{-mt^2}$ is a decreasing function on $(0,L)$. Moreover this function takes the value $e^{-1}$ if $t=m^{-\frac{1}{2}}$. We deduce that
	\begin{equation*}
	\int_{-L}^L e^{-m t^2} \d t \geqslant e^{-1} m^{-\frac{1}{2}},
	\end{equation*}
	and the lemma is proved. One can refer to the figure \ref{figure_minoration_integrale} for a visual intuition.
	
\begin{center}
\def\svgscale{0.7}
\begingroup%
  \makeatletter%
  \providecommand\color[2][]{%
    \errmessage{(Inkscape) Color is used for the text in Inkscape, but the package 'color.sty' is not loaded}%
    \renewcommand\color[2][]{}%
  }%
  \providecommand\transparent[1]{%
    \errmessage{(Inkscape) Transparency is used (non-zero) for the text in Inkscape, but the package 'transparent.sty' is not loaded}%
    \renewcommand\transparent[1]{}%
  }%
  \providecommand\rotatebox[2]{#2}%
  \ifx\svgwidth\undefined%
    \setlength{\unitlength}{419.84559066bp}%
    \ifx\svgscale\undefined%
      \relax%
    \else%
      \setlength{\unitlength}{\unitlength * \real{\svgscale}}%
    \fi%
  \else%
    \setlength{\unitlength}{\svgwidth}%
  \fi%
  \global\let\svgwidth\undefined%
  \global\let\svgscale\undefined%
  \makeatother%
  \begin{picture}(1,0.47025185)%
    \put(0,0){\includegraphics[width=\unitlength,page=1]{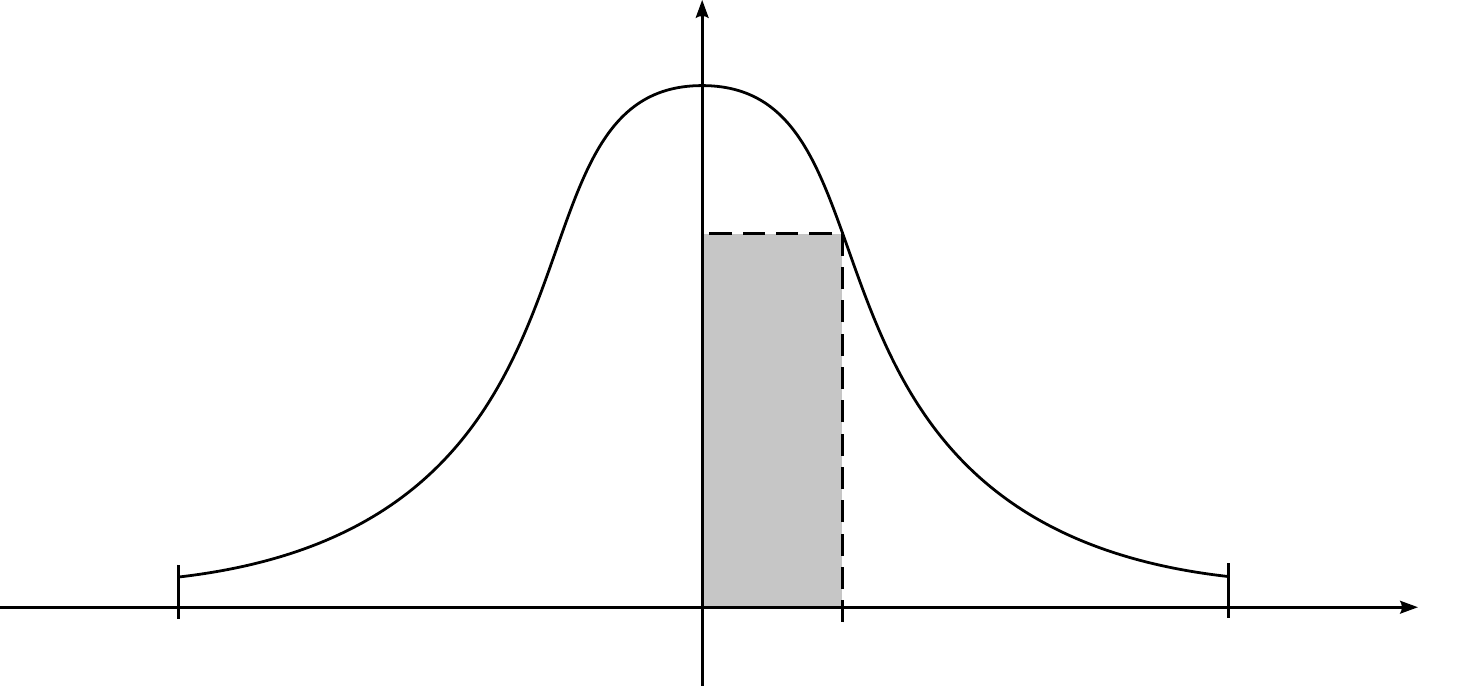}}%
    \put(0.09248671,0.01243718){\color[rgb]{0,0,0}\makebox(0,0)[lb]{\smash{$-L$}}}%
    \put(0.82840464,0.01375544){\color[rgb]{0,0,0}\makebox(0,0)[lb]{\smash{$L$}}}%
    \put(0.53872103,0.00836592){\color[rgb]{0,0,0}\makebox(0,0)[lb]{\smash{$m^{-\frac{1}{2}}$}}}%
    \put(0.6525734,0.16533403){\color[rgb]{0,0,0}\makebox(0,0)[lb]{\smash{$e^{-mt^2}$}}}%
    \put(0.4215002,0.29198638){\color[rgb]{0,0,0}\makebox(0,0)[lb]{\smash{$e^{-1}$}}}%
    \put(0.96381482,0.01779751){\color[rgb]{0,0,0}\makebox(0,0)[lb]{\smash{$t$}}}%
  \end{picture}%
\endgroup%

\captionof{figure}{\label{figure_minoration_integrale}Minoration of the integral by the area of the rectangle}
\end{center}
	
\end{proof}


Thanks to this lemma, we finally obtain
\begin{equation*}
\lambda_{1}(\tilde{\rho}_{m},\tilde{\rho}_{m}^{\alpha})
\geqslant
\frac{e^{-n(1-\alpha)}}{|\Omega|^{1-\alpha}} m^{1-\frac{n}{2}(1-\alpha)}.
\end{equation*}

	Since $\alpha > \frac{n-2}{n}$, then $1-\frac{n}{2}(1-\alpha) > 0$ which means that $\lambda_{1}(\tilde{\rho}_{m},\tilde{\rho}^{\alpha}_{m})\underset{m\to\infty}{\longrightarrow} \infty$ and the proof is complete.
\end{proof}
In the proposition below, we will show that in dimension $1$, the previous result
(Theorem~\ref{thm_sup_infini}) holds true for $\alpha\in (0,1)$.

\begin{proposition}
\label{minoration_dimension_1}

Let us take $M = (-1,1)$. Then for all $\alpha \in (0,1)$,
\[
\lambda_{1,\alpha}^{*}(M) = +\infty.
\]
\end{proposition}
\begin{proof} %
	
The above equation 
\begin{equation*}
-\div(\rho^{\alpha} \nabla u)=\lambda\rho u
\end{equation*}
with Neumann boundary conditions on $(-1,1)$ becomes
	\[
	\rho^{\alpha - 1}u'' + \alpha \rho^{\alpha - 2}\rho'u' + \lambda u = 0
	\]
	\[
	\rho^{\alpha - 1}u'' + \frac{\alpha}{\alpha-1}\left(\rho^{\alpha-1}\right)'u' + \lambda u = 0.
	\]
	We can differentiate to obtain
	\begin{equation*}
	\label{equation_premiere}
	\rho^{\alpha-1}u'''
	+\left(\rho^{\alpha-1}\right)'u''
	+\dfrac{\alpha}{\alpha-1}\left(\rho^{\alpha-1}\right)'u''
	+\dfrac{\alpha}{\alpha-1}\left(\rho^{\alpha-1}\right)''u'
	+\lambda u' =0.
	\end{equation*}
	Now we define $y=u'$ and the equation \eqref{luc3} becomes
	\[
	\left\{
	\begin{array}{lllll}
	\rho^{\alpha-1}y''
	&+&\dfrac{2\alpha-1}{\alpha-1}\left(\rho^{\alpha-1}\right)'y'
	&+&\left(\dfrac{\alpha}{\alpha-1}\left(\rho^{\alpha-1}\right)''
	+\lambda\right)y=0 \quad\text{in}\quad(-1,1)\\
	\\
	y(1)&=&\quad y(-1)&=& 0 
	\end{array}
	\right.
	\]
	
	Remark that if we multiply by $\rho^{\alpha}$ we get
	\[
	\rho^{2\alpha-1}y''
	+(2\alpha-1)\rho^{2\alpha-2}\rho'y'
	+\left(\dfrac{\alpha}{\alpha-1}\left(\rho^{\alpha-1}\right)''
	+\lambda\right)\rho^{\alpha}y =0,
	\]
	\textit{i.e.}
	\begin{equation}
	\label{eq_intervalle}
	\left(\rho^{2\alpha-1}y'\right)'
	+\left(\lambda-\dfrac{\alpha}{1-\alpha}\left(\rho^{\alpha-1}\right)''\right)\rho^{\alpha}y =0.
	\end{equation}
Now let $m$ be a positive integer. We want to choose $\rho=\rho_m$ such that $\lambda_1\left(\rho_m, \rho_m^{\alpha}\right) \geqslant m$. To do this we first solve the equation $\frac{\alpha}{1-\alpha}\left(\rho^{\alpha-1}\right)'' = m$ which admits (at least) one positive solution on $(-1,1)$. We choose the density such that
$\rho_m(x)^{\alpha-1}
	= \frac{1-\alpha}{2\alpha}\left(1+mx^2\right)$, that is,
	\[
		\rho_m(x)
		= \left(\frac{2\alpha}{1-\alpha}\right)^{\frac{1}{1-\alpha}}\dfrac{1}{\left(1+mx^2\right)^{\frac{1}{1-\alpha}}}.
	\]
	
	Now we multiply the equation \eqref{eq_intervalle} by $y$ and integrate it by parts to get
	\[
	\int_{-1}^1\left(\rho_m^{2\alpha-1}y'\right)'y \d x
	+\left(\lambda-m\right)\int_{-1}^1\rho_m^{\alpha}y^2\d x = 0,
	\]
	\[
	\left[\rho_m^{2\alpha-1}y'y\right]_{-1}^1 - \int_{-1}^1 \rho_m^{2\alpha-1}(y')^2\d x
	+\left(\lambda-m\right)\int_{-1}^1\rho_m^{\alpha}y^2\d x = 0.
	\]
	But we know that $y=u'$ vanishes at $-1$ and $1$. We obtain the following	
	\[
	\left(\lambda-m\right)\int_{-1}^1\rho_m^{\alpha}y^2 \d x
	= \int_{-1}^1 \rho_m^{2\alpha-1}(y')^2 \d x\geqslant 0.
	\]
	\begin{equation*}
		\text{So } \lambda=\lambda\left(\rho_m,\rho^{\alpha}_m\right) \geqslant m.
	\end{equation*}
	
	Again, we use the idea of Lemma \ref{minoration_integrale} to see that the normalised eigenvalue is not bounded:
	\begin{equation*}
	\lambda_{1}(\tilde{\rho}_{m},\tilde{\rho}_{m}^{\alpha})
	\geqslant
	m\cdot m^{-\frac{1}{2}} = m^{\frac{1}{2}}.
	\end{equation*}

%
%
	
	\bigskip
	The number $m$ being arbitrarily large, this concludes the proof.
\end{proof}

\addcontentsline{toc}{chapter}{Bibliographie}
\bibliographystyle{alpha}
\bibliography{biblio_densities}
\label{lastpage}
\end{document}